\documentclass[12pt,english]{article}

\usepackage{amssymb}
\usepackage{polski}
\usepackage{ifthen}
\selecthyphenation{polish}
\mathchardef\ogon="012C%
\newcommand{\as}{a\kern-0.22em\lower.40ex\hbox{$_{\ogon}$}}
\newcommand{\As}{A\kern-0.22em\lower.40ex\hbox{$_{\ogon}$}}
\newcommand{\es}{e\kern-0.24em\lower.40ex\hbox{$_{\ogon}$}}
\newcommand{\Es}{E\kern-0.22em\lower.40ex\hbox{$_{\ogon}$}}
\makeatletter
\usepackage[a4paper, left=2.5cm, right=2.5cm, top=2.5cm, bottom=3.5cm, headsep=1.2cm]{geometry}
\newtheorem{theorem}{Theorem}[section]

\newtheorem{corollary}[theorem]{Corollary}

\newtheorem{definition}[theorem]{Definition}
\newtheorem{example}[theorem]{Example}

\newtheorem{lemma}[theorem]{Lemma}

\newtheorem{proposition}[theorem]{Proposition}
\newtheorem{remark}[theorem]{Remark}

\newenvironment{proof}[1][Proof]{\noindent\textbf{#1.} }{\ \rule{0.5em}{0.5em}}
\def\qed{\hbox to 0pt{}\hfill$\rlap{$\sqcap$}\sqcup$}
\makeatother
\newcommand{\hilbertX}{{\mathcal X}}
\newcommand{\hilbertH}{{\cal H}}

\newcommand{\hilbertY}{{\mathcal Y}}
\newcommand{\dom}{{\text{dom\,}}}
\newcommand{\lagrangianL}{{\mathcal L}}
\newcommand{\barR}{\bar{\mathbb{R}}}
\usepackage{amssymb}
\usepackage{amsmath}
\usepackage{amsfonts}
\usepackage{color}
\usepackage{graphicx}
\usepackage{graphics}
\numberwithin{equation}{section}
\usepackage{babel}
\date{}
\title{On  duality for nonconvex minimization problems within the framework of abstract convexity }
\author{Ewa Bednarczuk\thanks{Systems Research Institute, Polish Academy of Sciences, Newelska 6, 01–447 Warsaw} \ \ and Monika Syga\thanks{Warsaw University of Technology, Faculty of Mathematics and Information Science, ul. Koszykowa 75,
		00--662 Warsaw, Poland,  M.Syga@mini.pw.edu.pl}
}
\begin{document}
	\maketitle

\begin{abstract}%
By applying the perturbation function approach, we propose   the Lagrangian and the conjugate duals for  minimization problems of the sum of two, generally nonconvex, functions.  The main tools are the $\Phi$-convexity theory and  minimax theorems for $\Phi$-convex functions. We provide conditions ensuring zero duality gap and introduce  $\Phi$-Karush-Kuhn-Tucker conditions that characterize solutions to primal and dual problems. We also discuss the relationship between the dual problems introduced in the present investigation
and some conjugate-type duals existing in the literature.

%\KEYWORDS{deterministic inventory theory; infinite linear programming duality; 
%  existence of optimal policies; semi-Markov decision process; cyclic schedule}
%\MSCCLASS{Primary: 90B05; secondary: 90C40, 90C90}
%\ORMSCLASS{Primary: Inventory/production: deterministic multi-item;
%  secondary: dynamic programming/optimal control: deterministic 
%  semi-Markov; programming: infinite dimensional}
%\HISTORY{Received November 20, 2003; revised March 8, 2004, and March 26, 2004.}

\medskip{}

\textbf{Mathematics Subject Classification (2000)32F17; 49J52; 49K27; 49K35; 52A01}
%\ORMSCLASS{Primary: ; secondary: }

\end{abstract}
%%%%%%%%%%%%%%%%%%%%%%%%%%%%%%%%%%%%%%%%%%%%%%%%%%%%%%%%%%%%%%%%%%%%%%

% Samples of sectioning (and labeling) in MOOR.
% NOTE: (1) all section levels end with a period,

\section{Introduction}

Let $\hilbertX$  be  a real  vector space.
We consider  the minimization problem  of the form
\begin{equation}
	\label{problem}
	\tag{P}
	\text{Min}_{x\in \hilbertX}\ \ f(x)+g(x). 
\end{equation}
where  $f,g:\hilbertX\rightarrow (-\infty,+\infty]$.
\vspace{0.2cm}

\noindent
{\bf Our standing assumptions and notations are as follows.}
\begin{itemize}
	\item[(a)] $\Phi$ and $\Psi$ are  classes (closed under addition of real constants) of   real-valued functions $\varphi:\hilbertX\rightarrow \mathbb{R}$,  called elementary functions, of simple structure, e.g. affine, quadratic,  step functions. In the sequel, we put additional requirements on sets of elementary functions  of algebraic character  when needed, e.g. in some of constructions and theorems we assume that $\Phi$ is symmetric ($\Phi=-\Phi$), and/or $\Phi$ is additive, ($\Phi+\Phi\subset\Phi$).
	\item[(b)] A  function $f:X\rightarrow(-\infty,+\infty]$ is proper, i.e.,  the  domain of $f$ is nonempty, i.e.
	$$
	\text{dom}(f):=\{x\in X \ : \ f(x)<+\infty \}\neq \emptyset.
	$$
	and $g:X\rightarrow(-\infty,+\infty]$ is proper 
	and $\dom (f)\cap \dom (g)\neq\emptyset$. 
	%\item {\color{cyan} $f$ is $\Phi$-convex, i.e. %$f(x):=\sup\limits_{\varphi\in\Phi}\{\varphi(x)\mid \varphi\le f\}$, 
	%$\varphi\le f\Leftrightarrow \forall_{x\in X}\ \varphi(x)\le f(x)$
	%\item {\color{cyan} $g$ is $\Psi$-convex, i.e. $g(x):=\sup\limits_{\psi\in\Psi}\{\varphi(x)\mid \psi\le g\}$}.}
	\item[(c)] When  $\hilbertX$  is  a Hilbert space  the inner product is denoted by $\langle\cdot\mid\cdot\rangle$ and the associated norm  is $\|\cdot\|$.
\end{itemize}

Main tool of the present investigation is the abstract convexity theory, called $\Phi$-convexity. The origins of the $\Phi$-convexity theory goes back to the investigations of Ky  Fan \cite{kyfan2}, Moreau \cite{moreau}, and Rubinov and Kutateladze \cite{KutRub72}. Applications in optimization were investigated e.g. by Balder  \cite{balder}, Dolecki and Kurcyusz \cite{dolecki-k}, Pallaschke \& Rolewicz \cite{rolewicz}, Rubinov \cite{rubbook}, in mass transport by R\"{u}schendorf \cite{Ruschendorf}. 

The underlying idea of $\Phi$-convexity also called {\em convexity without linearity}, is to replace the classical bi-linear coupling functions used in the convex analysis by general (possibly nonlinear) coupling functions. $\Phi$-convexity  provides a framework for the analysis of important classes of nonconvex problems. In the case of bi-linear coupling,  this framework   allows for the retrieval, and sometimes refinement, of classical results of convex analysis. 

$\Phi$-convexity  provides global tools for investigating nonconvex objects and  offers a framework for  investigating global optimization problems, see, e.g. the monographs by Alexander Rubinov  \cite{rubbook}, and by  Diethard Pallaschke \& Stefan Rolewicz \cite{rolewicz}. Basic concepts of $\Phi$-convexity theory from a historical perspective, its r\^ole in global optimization and duality theory have been recently discussed  in several presentations during  the on-line WOMBAT 2020 meeting (https://wombat.mocao.org/wombat-2020/recordings/).

An important class of $\Phi$-convex functions are $\Phi_{lsc}$-convex functions defined on a Hilbert space with elementary functions defined by \eqref{philsc}. In a series of papers, \cite{Rolewiczpara}, \cite{rolewicz1979}, \cite{rolewicz2000}, Stefan Rolewicz investigated particular subclass of $\Phi_{lsc}$-convex functions, called paraconvex functions (weakly convex, semiconvex) functions.

Formulae for $\Phi$-subdifferentials, $\Phi$-conjugates,  $\Phi$-infimal convolution  for  the sum  of two functions have been studied by  Jeyakumar, Rubinov \& Wu \cite{JRW} who also provided  generalizations of the results obtained for convex problems by Burachik \& Jeyakumar in \cite{burjey}. The results of \cite{JRW} have been generalized by Bui, Burachik, Kruger \& Yost in \cite{bui2020zero} to the  sum of any finite number of functions  $f_{i}$, $i=1,...,m$.  In \cite{bui2020zero} the the dual problem $(ICD)$ is formulated on the basis  of the $\Phi$-infimal convolution of $\Phi$-convex conjugates of functions $f_{i}$ in the classes of elementary functions $\Phi$ for which $0\in\Phi$ and $\Phi+\Phi\subset\Phi$.

In the present paper, we  construct  a $\Phi$-conjugate dual $(CD)$, where $0\in\Phi$,  for the problem of minimizing the sum of two proper functions which is based on the perturbation function $p(\cdot,\cdot)$.   We calculate the  $c$-conjugate $p_{c}^{*}(\cdot,\cdot)$ with respect to a suitably chosen coupling function $c$ and we define the dual  $(CD)$ as the problem of maximizing the function $-p^{*}_{c}(0,\cdot)$. This approach coincides with the approach to conjugate duality in the convex case, see e.g. \cite{bot, Bonnans}. 

We also introduce the $\Phi$-Lagrangian function ${\mathcal L}(\cdot,\cdot)$ for which the  $\Phi$-Lagrangian dual  \eqref{lagdual} is equivalent to the  $\Phi$-conjugate dual problem $(CD)$. 

The  question of conditions ensuring zero duality gap  is approached via minimax theorems  obtained  by Syga in \cite{syga} and \cite{Syga2018} for classes of elementary functions $\Phi$ which are convex sets. In Theorem 3.1 and Theorem 4.1 of \cite{syga},   the so-called {\em intersection property} is proved to be  a necessary and sufficient condition for the minimax equality to hold.  These results allow  us to prove  that the intersection property for the $\Phi$-Lagrange function is  necessary and sufficient  for zero duality gap both in the $\Phi$-Lagrangian  and the $\Phi$-conjugate dualities.

We investigate the relationships between $\Phi$-infimal convolution dual  $(ICD)$ and the  $\Phi$-conjugate dual problem $(CD)$ zero duality gap conditions obtained in \cite{bui2020zero}.

The contribution of the paper is as follows.
\begin{enumerate} 
	\item[-] Construction of the conjugate dual $(CD)$ via the conjugate $p_{c}^{*}(\cdot,\cdot)$ of the perturbation function $p$ with respect to the  coupling function $c$ (formula \eqref{cpl1}) (Section 3 and Section \ref{section4}).
	\item[-]  Construction of the Lagrangian dual equivalent to the conjugate dual (Section 5).
	\item[-] Derivation  of conditions for zero duality gap in the form of the so-called {\em intersection property} and discussion of its relationship to respective conditions for  $\Phi$-inf convolution-based zero duality gap proved in \cite{bui2020zero} (Section 6).
	\item [-] Definition of  $\Phi$-KKT conditions and   characterization of solutions to problems $(P)$ and $(CD)$. (Section 7).
\end{enumerate}

\section{Preliminaries}
\subsection{Abstract Convexity}
Let $\Phi$ be a set  of  elementary  real-valued functions $\varphi:\hilbertX\rightarrow \mathbb{R}$ and $f:\hilbertX\rightarrow(-\infty,+\infty]$. 
The set
$$
\text{supp}_{\Phi}(f):=\{\varphi\in \Phi\ :\ \varphi\le f\}
$$
is called the {\em support} of $f$ with  respect to $\Phi$, where, 
for any $g,h:\hilbertX\rightarrow(-\infty,+\infty]$,
$g\le h\ \Leftrightarrow\  g(x)\le h(x)\ \ \forall\ x\in \hilbertX.$
We will use the notation $\text{supp}(f)$ 
whenever the class $\Phi$ is clear from the context. Elements of  class $\Phi$ are called elementary functions.

\begin{definition}(\cite{dolecki-k}, \cite{rolewicz}, \cite{rubbook})
	\label{convf}
	A function $f:\hilbertX\rightarrow
	(-\infty,+\infty]$ is called {\em $\Phi$-convex on $\hilbertX$} if
	%	{\color{red} $\textnormal{supp}(f)\neq\emptyset$??? and}
	$$
	f(x)=\sup\{\varphi(x)\ :\ \varphi\in\textnormal{supp}(f)\}\ \ \forall\ x\in \hilbertX.
	$$
	If the set $\hilbertX$ is clear from the context, we simply say that $f$ is {\em $\Phi$-convex}.
\end{definition}

\subsection{$\Phi$-subgradients}

\begin{definition} [see e.g. formula 1.1.1 of \cite{rolewicz}  and Definition 1.7 of \cite{rubbook}]
	\label{def_subgradient_gen}
	An element  $\varphi\in\Phi$ is called a {\em $\Phi$-subgradient} of a  function  $f:\hilbertX\rightarrow (-\infty,+\infty]$ at $\bar{x}\in\text{dom\,} f$, if the following inequality holds
	\begin{equation}
		\label{phi_sub}
		f(x)-f(\bar{x}) \geq \varphi(x)-\varphi(\bar{x}) \ \ \forall \ x\in \hilbertX.\end{equation}
	The set of all  $\Phi$-subgradients of $f$ at  $\bar{x}$ is denoted as $\partial_{\Phi}  f(\bar{x})$.  For characterizations of subgradients see Proposition 1.2, Proposition 1.3 of \cite{rubbook} and formula $1.1.7_{max}$ of \cite{rolewicz}.
\end{definition}
%Definition \ref{def_subgradient_gen} was introduced in \cite{rolewicz} and \cite{rubbook}.

\begin{definition} [see e.g. Definition 7.8 of \cite{rubbook}, formula $1.1.1_{\varepsilon}$ of \cite{rolewicz}]
	An element  $\varphi\in\Phi$ is called a {\em $\Phi$-$\varepsilon$-subgradient} of a  function  $f:\hilbertX\rightarrow (-\infty,+\infty]$ at $\bar{x}\in\text{dom\,} f$, if the following inequality holds
	\begin{equation}
		\label{phi_sub1}
		f(x)-f(\bar{x}) \geq \varphi(x)-\varphi(\bar{x})-\varepsilon \ \ \forall \ x\in \hilbertX.
	\end{equation}
	The set of all  $\Phi$-$\varepsilon$-subgradients of $f$ at  $\bar{x}$ is denoted as $\partial^{\varepsilon}_{\Phi}  f(\bar{x})$.  If $f$ is $\Phi$-convex on $\hilbertX$, then, for every $\varepsilon>0$, the domain of the $\varepsilon$-subdifferential mapping $\text{dom}\partial^{\varepsilon}_{\Phi}  f=\{x\in\hilbertX\ \mid\ \partial^{\varepsilon}_{\Phi}  f(x)\neq\emptyset\}$ coincides with $\text{dom\,}f.$
	
\end{definition}

\subsection{$\Phi$-conjugation}
\label{conjugation}

Let $f:\hilbertX\rightarrow(-\infty,+\infty]$. The function $f^{*}_{\Phi}:\Phi\rightarrow(-\infty,+\infty]$,
\begin{equation}
	\label{conjugate} 
	f^{*}_{\Phi}(\varphi):=\sup_{x\in X} (\varphi(x)-f(x))
\end{equation}
is called the {\em $\Phi$-conjugate} of $f$ (see e.g.\cite{moreau}). The function $f^{*}_{\Phi}$ is convex whenever $\Phi$ is  convex.   For the characterisation of the epigraph of $f_{\Phi}^{*}$, see e.g. Proposition 7.8 of \cite{rubbook}.

Accordingly, the {\em $\Phi$-bi-conjugate} of $f$ is defined as
$$
f^{**}_{\Phi}(x):=\sup_{\varphi\in\Phi}(\varphi(x)-f^{*}_{\Phi}(\varphi)).
$$
The following relationships between $\varepsilon$-$\Phi$-subgradients, conjugate and biconjugate functions hold.

\begin{theorem} 
	\label{conju}
	Let $f:\hilbertX\rightarrow(-\infty,+\infty]$. 
	\begin{enumerate}
		\item[(i)] Fenchel-Moreau inequality: For every $x\in\text{dom\,}(f)$ and every $\varphi\in\Phi$
		\begin{equation} 
			\label{eq-eps_con1} 
			f(x)+f_{\Phi}^{*}(\varphi)\ge\varphi(x).
		\end{equation}
		\item[(ii)] For every $x\in\text{dom\,}(f)$ and every $\varepsilon>0$
		\begin{equation} 
			\label{eq-eps_con} 
			\varphi\in\partial_{\Phi}^{\varepsilon}f(x)\ \Leftrightarrow\ f(x)+f_{\Phi}^{*}(\varphi)\le\varphi(x)+\varepsilon.
		\end{equation}
		\item[(iii)] For every $x\in\text{dom\,}(f)$, 
		%\begin{equation} 
		%\label{eq-eps_con2} 
		$f_{\Phi}^{**}(x)\le f(x).$
		%\end{equation}
		\item[(iv)] For every  $x\in\text{dom}(f)$
		\begin{equation} 
			\label{eq-eps_con3} 
			(\forall\ \varepsilon>0)\ \partial_{\Phi}^{\varepsilon} f(x)\neq\emptyset\ \Rightarrow\ f(x)=f_{\Phi}^{**}(x).
		\end{equation}
		\item [(v)]  
		%\begin{equation} 
		%\label{eq-eps_con4}
		$\mbox{f is $\Phi$-convex on $\hilbertX$}\ \Leftrightarrow \ \ f(x)=f^{**}_{\Phi}(x)\ \ \forall \ \ x\in \hilbertX.$
		%	\end{equation}	
	\end{enumerate}
	
\end{theorem}

\begin{proof}
	\begin{enumerate} 
		\item[(i)] Follows directly from the definition of the $\Phi$-conjugate function, see e.g. Proposition 1 \cite{oettli} and Proposition 1.2.2 of \cite{rolewicz}.
		\item[(ii)] Follows from the definition of the conjugate and the $\varepsilon$-$\Phi$-subgradient, see Propositioon 1 of \cite{oettli}, Proposition 1.2.$4_{\varepsilon}$ of \cite{rolewicz} , Proposition 7.10 of \cite{rubbook}.
		\item[(iii)] Follows directly from Fenchel-Moreau inequality $(i)$.
		\item [(iv)] Let $\varepsilon>0$ and $\varphi\in\partial_{\Phi}^{\varepsilon}f(x)$. By $(ii)$ and the definition of the $\Phi$-bi-conjugate
		$$
		f(x)\le\varphi(x)-f_{\Phi}^{*}(x)+\varepsilon\le f_{\Phi}^{**}+\varepsilon.
		$$
		Since $\varepsilon$ is arbitrary, $f(x)\le f_{\Phi}^{**}(x)$. By $(iii)$, the conclusion follows.
		\item [(v)] For the proof see  Theorem 1.2.6 of \cite{rolewicz} and  Theorem 7.1 of \cite{rubbook},
	\end{enumerate}
\end{proof}

As noted in Proposition 1.2.3 of \cite{rolewicz}, 
the space $\hilbertX$  induces on $\Phi$ the family of functions $x:\Phi\rightarrow\mathbb{R}$ defined as $x(\varphi)=\varphi(x).$ This family of functions is also denoted by $\hilbertX$.
Hence, for any $\bar{x}\in\hilbertX$, $\bar{\varphi}\in\Phi$, we write $\bar{x}\in\partial_{\hilbertX}f_{\Phi}^{*}(\bar{\varphi})$ if
\begin{equation}
	\label{subgradientx}
	f_{\Phi}^{*}(\varphi)-f_{\Phi}^{*}(\bar{\varphi})\ge \varphi(\bar{x})-\bar{\varphi}(\bar{x})\ \ \ \ \text{for all  } \varphi\in\Phi.
\end{equation}
%For any function $f:\hilbertX\rightarrow\hat{\mathbb{R}}$   the Fenchel-Moreau inequality holds.  i.e.
%$$
%f(x)+f^*_{\Phi}(\varphi)\geq \varphi(x), \ \ \forall \ \ \varphi\in\Phi, \ x\in \hilbertX.
%$$
The following proposition holds (see also Proposition 1 of \cite{oettli}).
\begin{proposition}
	\label{you}
	Let $f:\hilbertX\rightarrow(-\infty,+\infty]$ be a $\Phi$-convex function.	Let $\bar{\varphi} \in \Phi$ and $\bar{x}\in \hilbertX$. The following conditions are equivalent.
	\begin{description}
		\item{(i)} $
		f(\bar{x})+f^*_{\Phi}(\bar{\varphi})=\bar{\varphi}(x)
		$.
		\item{(ii)} $\bar{\varphi} \in \partial_{\Phi}f(\bar{x})$.
		\item(iii)  $\bar{x}\in\partial_{\hilbertX} f^*(\bar{\varphi})$.
	\end{description}
	
\end{proposition}
\begin{proof}
	The equivalence between $(i)$ and $(ii)$ was proved in \cite{rubbook}, Proposition 7.7 and \cite{rolewicz}, Proposition 1.2.4. We show the equivalence between $(i)$ and $(iii)$.
	
	Assume that $(i)$ holds. Then, by $\Phi$-convexity of $f$, and Theorem \ref{conju}$(v)$, we get
	$$
	f^{**}_{\Phi}(\bar{x})+f^*_{\Phi}(\bar{\varphi})=\bar{\varphi}(x),
	$$
	which is equivalent to
	$$
	\varphi(x)-f^{*}_{\Phi}(\varphi)+f^*_{\Phi}(\bar{\varphi})\leq \bar{\varphi}(x) \ \ \ \ \forall \ \ \ \varphi \in \Phi,
	$$
	i.e.  $\bar{\varphi}\in \partial_{X}f_{\Phi}^{*}(\bar{x})$.
	
	Assume that $(iii)$ holds. We have the following inequality
	$$
	\bar{\varphi}(\bar{x}) \geq \varphi(\bar{x})-f^{*}_{\Phi}(\varphi)+f^*_{\Phi}(\bar{\varphi})  \ \ \ \ \forall \ \ \ \varphi \in \Phi.
	$$
	Taking the supremum over $\varphi\in \Phi$ we obtain
	$$
	\bar{\varphi}(\bar{x}) \geq \sup_{\varphi\in\Phi}\{\varphi(\bar{x})-f^{*}_{\Phi}(\varphi)\}+f^*_{\Phi}(\bar{\varphi})  \ \ \ \ \forall \ \ \ \varphi \in \Phi,
	$$
	which is equivalent to
	$$
	\bar{\varphi}(\bar{x}) \geq f(\bar{x}) +f^*_{\Phi}(\bar{\varphi})  \ \ \ \ \forall \ \ \ \varphi \in \Phi.
	$$
	This, together with the Fenchel-Moreau inequality, gives $(i)$.
\end{proof}

\section{Perturbation  function and its conjugate.}

Let $\hilbertX$ and $\hilbertY$ be  real linear vector spaces.
The perturbation function $p:\hilbertX\times\hilbertY\rightarrow(-\infty,+\infty]$ related to problem $(P)$
is defined  as
\begin{equation}
	\label{funkcjaper}
	p(x,y):= f(x) + g(x+y).
\end{equation}
Clearly, $p(x,0)=f(x)+g(x)$. 
%Observe that theroles , the r\^oles of $f$ and $g$ are no longer symmetric. 

In this section we investigate the conjugate $p^{*}$ to the perturbation function $p$. The obtained formulae, will be used in Section \ref{section4} to define the conjugate dual to  problem $(P)$.

%where $f:\hilbertX\rightarrow\barR$ is $\Phi$-convex on $\hilbertX$, $g:\hilbertX\rightarrow\barR$ is $\Psi$-convex on $\hilbertX$. Below we discuss two cases $\Phi=\Psi$ and $\Psi\subset \Phi$.

Let $\Phi$ and $\Psi$ be two classes of elementary functions defined on $\hilbertX$ and $\hilbertY$, respectively, where $\hilbertH$ is a real linear subspace, $\hilbertX+\hilbertY\subset \hilbertH$.
To define the conjugate $p^{*}$ we introduce the coupling function $c$ on the Cartesian product $\Phi\times\Psi$, 
$c:(\Phi\times\Psi)\times(\hilbertX\times \hilbertY)\rightarrow\mathbb{R}$
as follows
\begin{equation}
	\tag{$cpl$}
	\label{cpl1}
	c((\varphi,\psi), (x,y)):=\varphi(x)+\psi(x+y)-\psi(x).
\end{equation}
In other words, elementary functions $c=(\varphi,\psi):\hilbertX\times \hilbertY\rightarrow\mathbb{R}$ defined on the product $(\Phi\times\Psi)$ are of the form

\begin{equation}
	\tag{$c$}
	\label{c1}
	c=(\varphi,\psi)(x,y):=\varphi(x)+\psi(x+y)-\psi(x).
\end{equation}
By $(\Phi\times \Psi)^{c}$ we denote the Cartesian product
$\Phi\times \Psi$ equipped with the coupling \eqref{cpl1}.  For other  coupling functions defined on the Cartesian product  $\Phi\times \Psi$  see \cite{oettli}.

Clearly, if the class $\Psi$ consists of all affine functions, i.e. $\psi(y)=\langle w,y\rangle +d$, $w\in\hilbertX^{\circ}$ and $\hilbertX^{\circ}$ is the algebraic dual of $\hilbertX$, $d\in\mathbb{R}$, then 
$$
c((\varphi,\psi), (x,y)):=\varphi(x)+\psi_{0}(y),
$$
where $\psi_{0}(y)=\langle w,y\rangle$, which is the standard bi-linear coupling, see e.g. \cite{bot,Bonnans}. 

The  conjugate $p_{c}^{*}:(\Phi\times \Psi)^{c}\rightarrow(-\infty,+\infty]$  with respect to the coupling \eqref{cpl1}, i.e. with respect to the set of elementary functions $(\Phi\times \Psi)^{c}$, is given as
\begin{equation}
	\label{conjp}
	p_{c}^*(\varphi,\psi):=\sup\limits_{x,y\in \hilbertX}\{ \varphi(x)+\psi(x+y)-\psi(x)-p(x,y)\}.
\end{equation}

Clearly, the  conjugate $p_{c}^{*}$ depends  on the choice of the coupling between the Cartesian products $\hilbertX\times\hilbertY$ and $\Phi\times\Psi$ ( for other definitions of nonlinear couplings see e.g. \cite{oettli}.) 
In the sequel we simplify the notation and put $p^{*}:=p_{c}^{*}$

Now,
$$
p^{*}(\varphi,\psi)\begin{array}[t]{l}
	=\sup\limits_{x\in\hilbertX}\sup\limits_{y\in\hilbertX}\{\varphi(x)+\psi(x+y)-\psi(x)-f(x)-g(x+y)\}\\
	=\sup\limits_{x\in\hilbertX}\varphi(x)-f(x)-\psi(x)+\sup\limits_{y\in\hilbertX}\{\psi(x+y)-g(x+y)\}.
\end{array}
$$
By putting $z:=x+y$ we obtain 
\begin{equation} 
	\label{conjmain}
	p^*(\varphi,\psi)\begin{array}[t]{l}
		=\sup\limits_{x\in \hilbertX}\varphi(x)-f(x)-\psi(x) +\sup\limits_{z\in \hilbertX}\{ \psi(z)-g(z)\}\\
		=\sup\limits_{x\in \hilbertX}\varphi(x)-f(x)-\psi(x) +g_{\Psi}^{*}(\psi).
	\end{array}
\end{equation}

%In this section we consider the following class $\Phi$, 
%\begin{equation} 
%\label{philsc} 
%\Phi:=\Phi_{lsc}:=\{\varphi:\hilbertX\rightarrow\mathbb{R}\ \mid\ \varphi(x):=-a\|x\|^{2}+\langle %v,x\rangle+c,\ a\ge 0,\ v\in\hilbertX,\ c\in\mathbb{R}\}
%\end{equation}

\begin{enumerate}
	\item[(P1)] When  $\Psi=\Phi$, $0\in\Phi$, and $\hilbertX=\hilbertY$,
	%and both  $f,g:\hilbertX\rightarrow\barR$ are $\Phi$-convex on $\hilbertX$,
	by \eqref{conjmain}, for any $\varphi\in\Phi$
	\begin{equation} 
		\label{conjmain0}
		p^{*}(0,\varphi)\begin{array}[t]{l}
			=\sup\limits_{x\in\hilbertX}-f(x)-\varphi(x)+g_{\Phi}^{*}(\varphi)\   \\
			={}^{*}f_{\Phi}(\varphi)+g_{\Phi}^{*}(\varphi),
		\end{array}
	\end{equation}
	where ${}^{*}f_{\Phi}(\varphi):=\sup\limits_{x\in\hilbertX}-f(x)-\varphi(x)$. Clearly, ${}^{*}f_{\Phi}(\varphi)=f_{\Phi}^{*}(-\varphi)$ whenever $-\varphi\in\Phi$. When  $\Phi$ is a convex set, then $p^{*}(0,\cdot): \Phi\rightarrow(-\infty,+\infty]$ is convex.
	\item[(P2)] When $\hilbertX=\hilbertY$ is a Hilbert space and $\Phi=\Phi_{lsc}=\Psi$, where
	\begin{equation} 
		\label{philsc} 
		\Phi_{lsc}:=\{\varphi:\hilbertX\rightarrow\mathbb{R}\ \mid\ \varphi(x):=-a\|x\|^{2}+\langle v,x\rangle+c,\ a\ge 0,\ v\in\hilbertX,\ c\in\mathbb{R}\},
	\end{equation}
	by \eqref{conjmain}, for any $\varphi\in\Phi$, 
	\begin{equation} 
		\label{conjmain1}
		p^*(0,\varphi)\begin{array}[t]{l}
			=\sup\limits_{x\in \hilbertX}-f(x)+a\|x\|^{2}-\langle v,x\rangle -c +g_{\Phi_{lsc}}^{*}(\varphi)\\
			=\sup\limits_{x\in \hilbertX}-f(x)+a\|x\|^{2}-\langle v,x\rangle +g_{\Phi_{lsc}}^{*}(\varphi_{1})\\
			%=p^{*}(0,\varphi_{1})={}^{*}f(\varphi_{1})+g_{\Phi}^{*}(\varphi_{1})\\
			=\sup\limits_{x\in \hilbertX}-\tilde{f}(x)-\langle v,x\rangle +g_{\Phi_{lsc}}^{*}(\varphi_{1})\\
			=\tilde{f}_{\Phi_{lsc}}^{*}(0,-v) +g_{\Phi_{lsc}}^{*}(a,v)=p^{*}(0,\varphi_{1})\\
		\end{array}
	\end{equation}
	where $\varphi_{1}(x):=-a\|x\|^{2}+\langle v,x\rangle$ and $\tilde{f}(x):=\tilde{f}_{\varphi_{1}}(x)=f(x)-a\|x\|^{2}$ and we identify functions from the class $\Phi_{lsc}$ of the form $\varphi_{1}$ with pairs $(a,v)$, $a\in \mathbb{R}_{+}$, $v\in\hilbertX$. 
	By \eqref{conjmain1},    the domain of $p^{*}(0,\cdot)$ can be restricted to elementary  functions of the form $\varphi_{1}$ (with $c=0$). 
	As previously, 
	${}^{*}f_{\Phi_{lsc}}(\varphi_{1}):=\sup\limits_{x\in\hilbertX} -\varphi_{1}(x)-f(x)$. Clearly,
	${}^{*}f(\varphi_{1})={}^{*}f(a,v)=\tilde{f}^{*}(0,-v)$.
	By Proposition 6.3 of \cite{rubbook}  $\tilde{f}$ is $\Phi_{lsc}$-convex whenever $f$ is.
	Moreover, $p^{*}(0,\cdot)$ is a convex function on $\Phi_{lsc}$.

	\item[(P3)] 
	%When  $\Psi\subset \Phi$, 
	%the formula \eqref{conjmain0} remains true with $\psi\in\Psi$.  When $\Psi$ is a convex set, $p^{*}(0,\cdot)$ is a convex function on $\Psi$. 
	When $\Phi$ is symmetric, i.e., $\Phi=-\Phi$ with $(-\varphi)(x):=-\varphi(x)$ we have
	\begin{equation} 
		\label{psym}
		p^{*}(0,\varphi)=-f_{\Phi}^{*}(-\varphi)-g_{\Phi}^{*}(\varphi).
	\end{equation}
	Consider now $\Phi=\Phi_{conv}$, where $\hilbertX$ is a Banach space with the dual $\hilbertX^{*}$,
	\begin{equation} 
		\label{psiconv}
		\Phi_{conv}:=\{\varphi:\hilbertX\rightarrow \mathbb{R}\ |\ \varphi(x)=\langle v,x\rangle+c, \ v\in\hilbertX^{*},\ c\in\mathbb{R}\}.
	\end{equation}
	%i.e. $g$ is a proper lower semicontinuous convex function on $\hilbertX$.
	By \eqref{conjmain0}, 
	\begin{equation} 
		\label{conj1a}
		p^*(0,\varphi)\begin{array}[t]{l}
			%=\sup_{x\in \hilbertX}-f(x)-\langle v, x\rangle -c +g_{\Phi
			%-{conv}}^{*}(\varphi)\\
			=\sup_{x\in \hilbertX}-f(x)-\langle v, x\rangle -c +c+g^{*}(v)\\
			=f^{*}(-v) +g^{*}(v)=p^*(0,\varphi_{0}),
		\end{array}
	\end{equation}
	where $g^{*}(v):=\sup_{y\in\hilbertX}\langle v,y\rangle-g(y)$ is the conjugate to $g$ in the sense of convex analysis, $\varphi_{0}(y):=\langle v,y\rangle$.  By \eqref{conj1a},   we can  restrict the domain of $p^{*}(0,\cdot)$ to linear functionals.
\end{enumerate}

\section{The $\Phi$-conjugate dual}
\label{section4}

Following  the classical (convex) approach (see e.g.  Bo\c{t} \cite{bot} and Bonnans, Shapiro \cite{Bonnans}), we introduce
the  conjugate dual to $(P)$ by the formula
\begin{equation} 
	\label{gencondual}
	\tag{GCD}
	\text{Max\,}_{\psi\in\Psi} -p^*_{c}(0,\psi).
\end{equation}
%for  $\Psi=\hilbertX^{*}$.

\begin{enumerate}
	\item[(D1)] When $\hilbertX=\hilbertY$, $\Psi=\Phi$, $0\in\Phi$, and the coupling $c$ is given by \eqref{cpl1}, by \eqref{conjmain0}, 
	the  $\Phi$-conjugate dual  \eqref{gencondual} takes the form
	%with $\Psi=\hilbertX^{*}$ can be equivalently expressed as
	\begin{equation} 
		\label{condual3}
		\tag{CD}
		\text{Max\,}_{\varphi\in\Phi}-{}^{*}f_{\Phi}(\varphi)-g_{\Phi}^{*}(\varphi).
	\end{equation}
	where 
	${}^{*}f_{\Phi}(\varphi)=\sup\limits_{x\in\hilbertX}-f(x)-\varphi(x)$.  For dual problems resulting from other coupling functions $c$ see e.g. \cite{oettli}.
	%Note that the  $\Phi$-conjugate dual \eqref{condual3} is well defined for any class $\Phi$,  where $0\in\Phi$.

	%\begin{remark} Since the $\Phi$-conjugation is defined for any functions $f,g$,  the conjugate dual can be defined for $f,g$ which are not necessarily $\Phi$-convex. {\color{magenta}On the other hand, when proving e.g. zero duality gap results,  the $\Phi$-convexity of  $f$ and $g$ on $\hilbertX$ are crucial.}
	%\end{remark}
	
	\item [(D2)] When $0\in\Phi$  and the set $\Phi$ is symmetric, i.e. $\Phi=-\Phi$, problem \eqref{condual3} takes the form
	\begin{equation} 
		\label{condual3a}
		\tag{$CD^{sym}$}
		\text{Max\,}_{\varphi\in\Phi}-f_{\Phi}^{*}(-\varphi)-g_{\Phi}^{*}(\varphi).\end{equation}
	(c.f. Corollary 5.2 of \cite{JRW}.) 
	
	\item [(D3)] When  $0\in\Phi$,  and $\Phi+\Phi\subset\Phi$ and $\Phi=-\Phi$  the problem \eqref{condual3a}  becomes  the $\Phi$-infimal convolution dual  \eqref{hoabui} as introduced in \cite{bui2020zero}, 
	\begin{equation} 
		\label{hoabui}
		\tag{$ICD$}
		\text{Max\,}_{\varphi_{1},\varphi_{2}\in\Phi,\ \varphi_{1}+\varphi_{2}=0}-f_{\Phi}^{*}(\varphi_{2})-g_{\Phi}^{*}(\varphi_{1}).
	\end{equation}
	In general, when $\Phi$ is not symmetric we have
	\begin{equation}
		\label{cd}
		val\eqref{condual3}\geq val\eqref{condual3a}.
	\end{equation}
\end{enumerate}

\begin{example}
	\begin{enumerate} 
		\item Let $\hilbertX$ be a Hilbert space and $\Phi=\Phi_{lsc}$. 
		%by \eqref{conjmain1}, 
		The $\Phi_{lsc}$-conjugate dual  \eqref{condual3} takes the form
		\begin{equation} 
			\label{conjugate_dual1}
			\tag{$CD^{lsc}$}
			\text{Max\,}_{(a,w)\in\Phi_{lsc}} -(\tilde{f}_{(a,w)})_{\Phi_{lsc}}^{*}(0, -w)-g_{\Phi_{lsc}}^{*}(a,w),
		\end{equation}
		where  functions of the form $\varphi_{1}(x):=-a\|x\|^{2}+\langle w,x\rangle$  are identified with  pairs $(a,w)$, $w\in\hilbertX$, $a\in\mathbb{R}_{+}$ and $\tilde{f}_{\varphi_{1}}(x):=\tilde{f}_{(a,w)}(x)=f(x)-a\|x\|^{2}$ (according to \eqref{conjmain1} we can neglect  constants).  Since  $\Phi_{lsc}\neq-\Phi_{lsc}$, the $\Phi_{lsc}$-conjugate dual does not coincide, in general, with $\Phi_{lsc}$-infimal convolution dual \eqref{hoabui}, see Example \ref{example1} below.

		\item When $\hilbertX$ is a Banach space, and $\Phi=\hilbertX^{*}$, 
		%i.e. $f$ and $g$ are proper lower semicontinuous convex functions
		the $\hilbertX^{*}$-conjugate dual \eqref{condual3} becomes
		the classical  {\em Fenchel dual}
		\begin{equation} 
			\label{fencheldual}
			\tag{FD}
			\text{Max\,}_{v\in \hilbertX^{*}} -f^{*}(-v)-g^{*}(v).
		\end{equation}
		\eqref{fencheldual} coincides with \eqref{condual3a} and the $\hilbertX^{*}$-infimal convolution dual \eqref{hoabui}. 
	\end{enumerate}
	%{\color{cyan} More generally, when   $0\in\Phi$, $\Phi=-\Phi$  and $\Phi+\Phi\subset\Phi$, %then \eqref{condual3} coincides with \eqref{condual3a} and \eqref{hoabui}.}
\end{example}

\subsection{Weak (conjugate) duality}
Let $\hilbertX$ be a real linear space. 
%and $\Psi\subseteq\Phi$. 
%For every $x\in \hilbertX$, $\varphi\in\Phi$, $\psi\in\Psi$,
%$$
%p(x,y)+p^{*}(\varphi,\psi)\ge p(x,y)+\varphi(x)+\psi(x+y)-\psi(y)-p(x,y).
%
%If the class $\Psi$ has the property that $-\psi\in\Psi$ whenever $\psi\in\Psi$, then, 
By  \eqref{conjmain0}, for every $x\in \hilbertX$ and $\varphi\in\Phi$
$$
p(x,0)+p^{*}(0,\varphi)\begin{array}[t]{l}
	=p(x,0)+{}^{*}f(\varphi)+g_{\Phi}^{*}(\varphi),\\
	\ge f(x)+g(x)- f(x)-\varphi(x)+\varphi(x)-g(x)\\
	=0
\end{array}
$$
In consequence,
\begin{equation}
	\label{pdineq}
	p(x,0)\ge-p^{*}(0,\varphi) \ \ \ x\in \hilbertX\ \ \ \varphi\in\Phi
\end{equation}
which yields the {\em weak duality} 
%for $\Psi$ such that $\psi\in\Psi\Rightarrow \psi\in\Psi$
\begin{equation}
	\label{weak_duality} 
	val(P):=\inf\limits_{x\in \hilbertX}p(x,0)=\inf\limits_{x\in X} f(x)+g(x)\ge\sup\limits_{\varphi\in\Phi}- p^{*}(0,\varphi)=:val(CD).
\end{equation}
%When $0\in\Psi$, by\eqref{inequality},
%\begin{equation}
%\label{weak_duality1} 
%val(ICD)\le val (CD)\le val(P)
%\end{equation}
%When $\Phi=\Phi_{lsc}=\Psi$, by \eqref{conjmain1}, for any $x\in\hilbertX$, and $\psi\in\Psi_{lsc}$, we obtain
%\begin{equation}
%\label{pdineqlsc}
%p(x,0)+p^{*}(0,\psi) \begin{array}[t]{l}   
%=p(x,0)+p^{*}(0,\psi_{1})\\
%=f(x)+g(x)+\tilde{f}^{*}(0,-w) +g^{*}(a,w)\\
%\ge f(x)+g(x)-f(x)+a\|x\|^{2}-\langle w,x\rangle+\langle w,x\rangle-g(x)\\
%\ge 0
%\end{array}
%\end{equation}
%where $\psi_{1}(x):=-a\|x\|^{2}+\langle w,x\rangle$ is identified with pairs %$(a,w)$, $a\in \mathbb{R}_{+}$, $w\in\hilbertX$, %$\tilde{f}(x):=f(x)-a\|x\|^{2}$. This gives \eqref{weak_duality}. 

The problem of zero duality  gap will be addressed  in Section \ref{zeroduality}.

\section{Lagrangian dual}
\label{seclag}
In this section we introduce the $\Phi$-Lagrangian function \eqref{lagrange}  with the $\Phi$-Lagrangian dual  equivalent  to the $\Phi$-conjugate dual \eqref{condual3}.

%\subsection{$\Phi=\Psi$}
For problem \eqref{problem}, 
%where both $f$ and $g$ are $\Phi$-convex on $\hilbertX$, 
we consider  the $\Phi$-Lagrangian $\lagrangianL:\hilbertX\times\Phi\rightarrow\barR$ defined as

\begin{equation} 
	\label{lagrange}
	\tag{L}
	{\mathcal L}(x,\varphi):=f(x)+\varphi(x)-g_{\Phi}^{*}(\varphi)
\end{equation}
with the $\Phi$-Lagrangian primal  
\begin{equation} 
	\tag{LP}
	\label{lagprimal}
	\inf\limits_{x\in \hilbertX}\sup\limits_{\varphi\in\Phi}{\cal L}(x,\varphi)
\end{equation}
and the $\Phi$-Lagrangian dual 
\begin{equation} 
	\tag{LD}
	\label{lagdual}
	\sup\limits_{\varphi\in\Phi}\inf\limits_{x\in \hilbertX} {\mathcal L}(x,\varphi).
\end{equation}
Then
\begin{equation} 
	\label{lag_primal_eq}
	\sup\limits_{\varphi\in\Phi}{\mathcal L}(x,\varphi)
	%\begin{array}[t]{l}
	%=f(x)+\sup_{\psi\in\Psi}-b\|Ax\|^{2}+\langle w,Ax\rangle-g^{*}(\psi)\\
	=f(x)+\sup\limits_{\varphi\in\Phi}\varphi(x)-g_{\Phi}^{*}(\varphi)
	=f(x)+g_{\Phi}^{**}(x).
	%\end{array}
\end{equation}

\begin{proposition} 
	\label{lagprimalequality}
	If $g:\hilbertX\rightarrow\barR$ is $\Phi$-convex on $\hilbertX$,  the $\Phi$-Lagrangian primal \eqref{lagprimal}  is equivalent to  \eqref{problem},  i.e.,
	\begin{equation} 
		\label{eqprimal} 
		\inf\limits_{x\in \hilbertX} f(x)+g(x)=\inf\limits_{x\in \hilbertX}\sup\limits_{\varphi\in\Phi}{\mathcal L}(x,\varphi),
	\end{equation}
\end{proposition}
\begin{proof} Follows from Theorem \ref{conju} $(v)$.
\end{proof}

On the other hand, by \eqref{lagrange},
$$
\inf_{x\in \hilbertX}{\cal L}(x,\varphi)=\inf_{x\in \hilbertX}f(x)+\varphi(x)-g_{\Phi}^{*}(\varphi)=-\sup_{x\in \hilbertX}-f(x)-\varphi(x)-g_{\Phi}^{*}(\varphi).
$$
By using the notation  ${}^{*}f_{\Phi}(\varphi)=\sup\limits_{x\in\hilbertX}-\varphi(x)-f(x)$, 
\begin{equation}
	\label{lagconj}
	\sup\limits_{\varphi\in\Phi}\inf\limits_{x\in \hilbertX}{\mathcal L}(x,\varphi)=\sup\limits_{\varphi\in\Phi}-{}^{*}f_{\Phi}(\varphi)-g_{\Phi}^{*}(\varphi)
\end{equation}
which shows that  the $\Phi$-conjugate dual \eqref{condual3} is equivalent to  the $\Phi$-Lagrangian dual \eqref{lagdual} with the Lagrangian defined by \eqref{lagrange}.

%\subsection{$\Phi=\Phi_{lsc}$}
\begin{example}
	Let $\hilbertX$ be a Hilbert space. For $\varphi\in\Phi_{lsc}$, $\varphi(x):=-a\|x\|^{2}+\langle w,x\rangle$, $a\ge 0$, $w\in \hilbertX$, we have
	$$
	\inf\limits_{x\in \hilbertX}{\cal L}(x,\varphi)\begin{array}[t]{l}
		=\inf\limits_{x\in \hilbertX}f(x)+\varphi(x)-g_{\Phi_{lsc}}^{*}(\varphi)=-\sup\limits_{x\in \hilbertX}-f(x)-\varphi(x)-g_{\Phi_{lsc}}^{*}(\varphi)\\
		=-\sup\limits_{x\in \hilbertX}-f(x)+a\|x\|^{2}-\langle w,x\rangle-g_{\Phi_{lsc}}^{*}(\varphi)\\
		=-\sup\limits_{x\in \hilbertX}-\tilde{f}(x)-\langle w,x\rangle-g_{\Phi_{lsc}}^{*}(\varphi)\\
		=-\tilde{f}_{\Phi_{lsc}}^{*}(0,-w)-g^{*}_{\Phi_{lsc}}(a,w),
	\end{array}
	$$
	where $\varphi$ is identified with the pair $(a,w)$, and, for a given $\varphi$, $\tilde{f}:=f(\cdot)-a\|\cdot\|^{2}$.
	\begin{equation}
		\label{lagconjlsc}
		\sup\limits_{\varphi\in\Phi}\inf\limits_{x\in \hilbertX}{\mathcal L}(x,\varphi)=\sup\limits_{\varphi\in\Phi_{lsc}}-\tilde{f}_{\Phi_{lsc}}^{*}(0,-w)-g_{\Phi_{lsc}}^{*}(a,w),
	\end{equation}
	and  the $\Phi_{lsc}$-conjugate dual \eqref{conjugate_dual1} coincides with the $\Phi_{lsc}$-Lagrangian dual \eqref{lagdual}.
\end{example}

\begin{example}
	Let $\hilbertX$ be a Banach space. Let $\Phi:=\{\varphi:\hilbertX\rightarrow\mathbb{R}\mid \varphi(x):=\langle v,x\rangle+c, \ v\in\hilbertX^{*},\ c\in\mathbb{R}\}$ and let $g(\cdot):=\text{ind\,}_{A}(\cdot)$ be the indicator function of a set $A$,
	$$
	A:=\{x\in\hilbertX\mid x\in K\}
	$$
	where $K$ is a  cone in $\hilbertX$. For the problem $(P)$ with $f:\hilbertX\rightarrow(-\infty,+\infty]$
	%and $g:\hilbertX\rightarrow(-\infty,+\infty]$ the indicator function of $A$, $g:=\text{ind}_{A}$,
	$$
	{\mathcal L}(x,\varphi)\begin{array}[t]{l}
		= f(x)+ \varphi(x) -(\text{ind\,}_{A})^{*}(\varphi)\\
		=f(x)+\langle v,x\rangle+c-\sup\limits_{x\in\hilbertX} \{\langle v,x\rangle+c-\text{ind\,}_{K}(x)\}\\
		=f(x)+\langle v,x\rangle-\sup\limits_{x\in\hilbertX} \{\langle v,x\rangle-\text{ind\,}_{K}(x)\}\\
		=f(x)+\langle v,x\rangle-\sup\limits_{ x\in K } \langle v,x\rangle\\
	\end{array}
	$$
	Since
	$$
	\sup\limits_{ x\in K } \langle v,x\rangle=\left\{\begin{array}{ll}
		0&v\in K^{\circ}\\
		+\infty&v\not\in K^{\circ}\end{array}\right.
	$$
	where $K^{\circ}:=\{v\in\hilbertX\mid \langle v,x\rangle\le 0\ \ \forall\ x\in K\}$ is the polar cone to $K$, we get
	$$
	{\mathcal L}(x,\varphi)=\left\{\begin{array}{ll}
		f(x)+\langle v,x\rangle&v\in K^{circ}\\
		-\infty&v\not\in K^{\circ}\end{array}\right. .
	$$
	Consequently, the $\Phi$-Lagrangian primal \eqref{lagprimal} is
	$$
	\inf\limits_{x\in\hilbertX} \sup\limits_{v\in K^{*}} f(x)+\langle v,x\rangle=\inf\limits_{x\in\hilbertX} f(x)+\sup\limits_{v\in K^{\circ}}\langle v,x\rangle,
	$$
	and  the $\Phi$-Lagrangian dual \eqref{lagdual} is 
	$$
	\sup_{v\in K^{\circ}}\inf\limits_{x\in\hilbertX} f(x)+\langle v,x\rangle.
	$$
	is equivalent to the $\Phi$-conjugate dual \eqref{condual3} is
	$$
	\sup_{v\in K^{*}}- f^{*}(-v).
	$$

\end{example}

\section{Zero duality gap for $\Phi$-conjugate  duality}
\label{zeroduality}

In view of Proposition \ref{lagprimalequality}, and  formula \eqref{lagconj}, the question of zero duality gap for $\Phi$-conjugate and $\Phi$-Lagrangian dualities can be investigated simultaneously, by seeking  conditions 
%ensuring the equivalence between 
%problems \eqref{problem} and \eqref{gencondual}, or, by investigating conditions 
ensuring minimax equality for $\Phi$-Lagrangian.

We begin this section by discussing   zero duality gap for problems \eqref{lagprimal}, \eqref{lagdual} from the point of view of minimax theorems. The characterisation of zero duality gap for  problems \eqref{lagprimal}, \eqref{lagdual} is expressed with the help of the so called  {\em intersection property}, which is used in  general minimax theorems formulated within the framework of $\Phi$-convexity as it is done in \cite{Syga2018} for the  case, where the elementary functions may admit infinite values.
For convenience of the reader  we provide the outline of the proof  based on Lemma \ref{key_lemma}. The intersection property together with the condition $\inf\limits_{x\in\hilbertX} f(x)+g(x)=\inf\limits_{x\in\hilbertX} f(x)+g^{**}_{\Phi}(x)$ immediately gives the zero duality gap condition for the pair of dual problems \eqref{problem} and \eqref{condual3}.

\begin{theorem}
	\label{lagmin}
	Let $\hilbertX$ be a real vector space. Let
	$\Phi$ be a convex set of elementary functions $\varphi:\hilbertX\rightarrow \mathbb{R}$ and $f,g:\hilbertX\rightarrow(-\infty,+\infty]$ and the $\Phi$-Lagrangian is given by \eqref{lagrange}.
	
	The following are equivalent:
	\begin{description}
		\item[(i)]	for every $\alpha <\inf\limits_{x\in X}\sup\limits_{\varphi\in\Phi}{\mathcal L}(x, \varphi)$ there exist $\varphi_1,\varphi_2\in \Phi$ and  $ \bar{\varphi}_1 \in\text{supp} {\mathcal L}(\cdot,\varphi_1)$  and $ \bar{\varphi}_2 \in\text{supp} {\mathcal L}(\cdot,\varphi_2)$ such that functions $\bar{\varphi}_1$ and $\bar{\varphi}_1$ have the {\em intersection property} on $\hilbertX$ at the level $\alpha$, i.e., for all $t\in [0,1]$
		\begin{equation}
			\label{eq-11}
			[t\bar{\varphi}_1+(1-t)\bar{\varphi}_2<\alpha]\cap [\bar{\varphi}_1<\alpha]=\emptyset \ \ \ \ \text
			\ \ \ \ [t\bar{\varphi}_1+(1-t)\bar{\varphi}_2<\alpha]\cap [\bar{\varphi}_2<\alpha]=\emptyset,
		\end{equation}
		where $[\bar{\varphi}<\alpha]:=\{x\in\hilbertX \ \bar{\varphi}(x)<\alpha \}$.
		\item[(ii)]	$$\inf_{x\in \hilbertX}\sup_{\varphi\in\Phi}{\mathcal L}(x, \varphi)=\sup_{\varphi\in\Phi}\inf_{x\in \hilbertX}{\mathcal L}(x, \varphi).$$
	\end{description}
\end{theorem}

The proof of Theorem \ref{lagmin} is based on the following  lemma (c.f. \cite{Syga2018}, Lemma 4.1).

\begin{lemma} 
	\label{key_lemma}
	Let $X$ be a set, $\alpha\in\mathbb{R}$, and let $\varphi_{1},\ \varphi_{2}:X\rightarrow\mathbb{R}$ be any two functions.  The functions $\varphi_{1}$ and $\varphi_{2}$ have the intersection property on $X$ at the level $\alpha$ if and only if $\exists\ t_{0}\in[0,1]$ such that
	\begin{equation}
		\label{key_ineq} 
		t_{0}\varphi_{1}+(1-t_{0})\varphi_{2}\ge \alpha\ \ \forall\ x\in X.
	\end{equation}
\end{lemma}
Lemma 4.1  proved in \cite{Syga2018} refers to a more general situation, where $\varphi_{1},\ \varphi_{2}:X\rightarrow [-\infty,+\infty]$ and reduces to Lemma \ref{key_lemma} whenever $\varphi_{1},\ \varphi_{2}:X\rightarrow \mathbb{R}$.

\begin{proof}
	Let $\alpha <\inf\limits_{x\in X}\sup\limits_{\varphi\in\Phi}{\mathcal L}(x, \varphi)$. By $(i)$, there exist $\varphi_1,\varphi_2\in \Phi$ and  $ \bar{\varphi}_1 \in\text{supp} {\mathcal L}(\cdot,\varphi_1)$  and $ \bar{\varphi}_2 \in\text{supp} {\mathcal L}(\cdot,\varphi_2)$ such that $ \bar{\varphi}_1 \in\text{supp} {\mathcal L}(\cdot,\varphi_1)$  and $ \bar{\varphi}_2 \in\text{supp} {\mathcal L}(\cdot,\varphi_2)$ have the intersection property on $\hilbertX$ at the level $\alpha$. By  Lemma \ref{key_lemma} and \eqref{key_ineq}, there exists $t\in [0,1]$ such that
	\begin{equation}
		\label{key_ineq1} 
		t\bar{\varphi}_{1}+(1-t)\bar{\varphi}_{2}\ge \alpha\ \ \forall\ x\in X.
	\end{equation}
	By the definition of the support set and the inequality \eqref{key_ineq1} we get
	\begin{equation}
		\label{key_ineq2} 
		t{\mathcal L}(x,\varphi_1)+(1-t){\mathcal L}(x,\varphi_2)\ge \alpha\ \ \forall\ x\in X.
	\end{equation}
	By the concavity of ${\mathcal L}$ as a function of $\varphi$, we have
	\begin{equation}
		\label{key_ineq3} 
		{\mathcal L}(x,\varphi_0)\ge \alpha\ \ \forall\ x\in X,
	\end{equation}
	where $\varphi_0=t\varphi_1+(1-t)\varphi_2$ and, by convexity of $\Phi$, $\varphi_0\in\Phi$.
	
	From this we deduce the following inequality
	\begin{equation}
		\label{key_ineq4} 
		\sup\limits_{\varphi\in\Phi}  \inf\limits_{x\in X} {\mathcal L}(x,\varphi)\ge \alpha\ \ \forall\ x\in X.
	\end{equation}
	By the fact that the inequality  \eqref{key_ineq4} holds for every $\alpha <\inf\limits_{x\in X}\sup\limits_{\varphi\in\Phi}{\mathcal L}(x, \varphi)$  we get the desired conclusion.
	
	The second implication follows directly from Theorem 2.1 of \cite{Syga2018}.
	
\end{proof}

\begin{remark}
	Let us note that in some classes of functions (e.g. $\Phi_{lsc}$ and $\Phi_{conv}$, see Proposition 2 and Proposition 4 of \cite{bed-syg}) the intersection property at the level $\alpha$ is equivalent to the condition
	\begin{equation}
		\label{empty}
		[\varphi_1<\alpha]\cap [\varphi_2<\alpha]=\emptyset.
	\end{equation}
\end{remark}

Theorem \ref{lagmin} allows us to formulate the following zero duality gap conditions for $\Phi$-conjugate dual \eqref{condual3}.

\begin{theorem}
	\label{lagmin1}
	Let $\hilbertX$ be a real vector space. Let $\Phi$ be a convex set of elementary functions $\varphi:\hilbertX\rightarrow \mathbb{R}$,  $f,g:\hilbertX\rightarrow(-\infty,+\infty]$ and the $\Phi$-Lagrangian is given by \eqref{lagrange}.
	Assume that 
	\begin{equation} 
		\label{eq_equality} 
		\inf_{x\in\hilbertX} f(x)+g(x)=\inf_{x\in \hilbertX} f(x)+g^{**}_{\Phi}(x).
	\end{equation}
	The following are equivalent:
	\begin{description}
		\item[(i)]	for every $\alpha <\inf\limits_{x\in X}\sup\limits_{\varphi\in\Phi}{\mathcal L}(x, \varphi)$ there exist $\varphi_1,\varphi_2\in \Phi$ and  $ \bar{\varphi}_1 \in\text{supp} {\mathcal L}(\cdot,\varphi_1)$  and $ \bar{\varphi}_2 \in\text{supp} {\mathcal L}(\cdot,\varphi_2)$ such that functions $\bar{\varphi}_1$ and $\bar{\varphi}_1$ have the {\em intersection property} at the level $\alpha$, i.e., for all $t\in [0,1]$
		\begin{equation}
			\label{eq-111}
			[t\bar{\varphi}_1+(1-t)\bar{\varphi}_2<\alpha]\cap [\bar{\varphi}_1<\alpha]=\emptyset \ \ \ \ \text{or}
			\ \ \ \ [t\bar{\varphi}_1+(1-t)\bar{\varphi}_2<\alpha]\cap [\bar{\varphi}_2<\alpha]=\emptyset.
		\end{equation}
		\item[(ii)]	$$
		\inf\limits_{x\in \hilbertX}\{ f(x)+g(x)\}=
		%\inf_{x\in \hilbertX}\sup_{\varphi\in\Phi}{\mathcal L}(x, \varphi)=
		\sup_{\varphi\in\Phi} -{}^{*}f_{\Phi}(\varphi)-g^{*}_{\Phi}(\varphi),
		$$
	\end{description}
	where ${}^{*}f_{\Phi}(\varphi)=\sup\limits_{x\in\hilbertX} -\varphi(x)-f(x).$
\end{theorem}
\begin{proof}
	Follows directly from Theorem \ref{lagmin} and formula \eqref{lag_primal_eq}.
	%Proposition \ref{lagprimalequality}.
\end{proof}
\vskip 0.3 true in

By Proposition \ref{lagprimalequality} if $g$ is $\Phi$-convex on $\hilbertX$, then \eqref{eq_equality} holds.
The following conditions for zero duality gap for problems \eqref{problem} and \eqref{hoabui} were proved in \cite{bui2020zero}.

\begin{theorem}(\cite{bui2020zero}, Theorem 3.5)
	\label{bui3.6}
	Let $\hilbertX$ be a real linear vector space. Let $f,g:\hilbertX \rightarrow(-\infty, +\infty]$  and  $0\in \Phi$ and  $\Phi+\Phi\subset\Phi$. The following properties are equivalent:
	\begin{description}
		\item[(i)] 
		\begin{equation} 
			\label{bui}
			0\in \bigcap_{\varepsilon>0}(\partial^{\varepsilon}_{\Phi}f+ \partial^{\varepsilon}_{\Phi}g)(\hilbertX),
		\end{equation}
		
		\item[(ii)] $\inf\limits_{x\in \hilbertX}(f(x)+g(x))=\sup\limits_{\varphi_{1},\varphi_{2}\in \Phi,\ \varphi_{1}+\varphi_{2}=0}(-f_{\Phi}^*(\varphi_{1})-g_{\Phi}^*(\varphi_{2}))=val(ICD)<+\infty$.
	\end{description}
\end{theorem}

\begin{remark} 
	The  following inequalities hold
	\begin{equation} 
		\label{eq_grop}  
		val(P)\stackrel{(1)}{\ge} val (LP)\stackrel{(2)}{\ge} val (LD)\stackrel{(3)}{=}val(CD)\stackrel{(4)}{\ge} val(ICD), 
	\end{equation}
	where $(1)$ holds by Theorem \ref{conju} $(iii)$,  $(2)$ holds by general minimax inequality,
	$(3)$ holds by \eqref{lagconj} and $(4)$ holds by \eqref{cd}.

	In particular, condition \eqref{bui} of Theorem \ref{bui3.6} implies that  $val(P)=val(LP)$, i.e.
	$$
	\inf\limits_{x\in\hilbertX} f(x)+g(x)=\inf\limits_{x\in\hilbertX} f(x)+g_{\Phi}^{**}(x).
	$$
\end{remark}

\begin{theorem}
	\label{th_implication}	
	Let $\hilbertX$  be  a  real vector space,  $0\in\Phi$,  and $\kappa:=\inf\limits_{x\in \hilbertX} \sup\limits_{\varphi\in \Phi} {\mathcal L}(x,\varphi)<+\infty$.

	Consider the following conditions: 
	\begin{enumerate}
		\item 	condition \eqref{bui}:
		$$
		0\in \bigcap_{\varepsilon>0}(\partial^{\varepsilon}_{\Phi}f+ \partial^{\varepsilon}_{\Phi}g)(\hilbertX),
		$$
		\item condition \eqref{eq-111}:
		for every $\alpha< \inf\limits_{x\in\hilbertX} \sup\limits_{\varphi\in \Phi} {\mathcal L}(x,\varphi)$
		there exist $\bar{\varphi}, \tilde{\varphi} \in \Phi$, $\varphi_1\in \text{supp}{\cal L}(,\cdot,\bar{\varphi} ) $ and $\varphi_2\in \text{supp}{\cal L}(,\cdot,\tilde{\varphi} ) $ such that $\varphi_1$ and $\varphi_2$ have the intersection property at the level $\alpha$.
	\end{enumerate}
	Then
	\begin{enumerate}
		\item If $\Phi$ is convex, $\Phi=-\Phi$, and 
		\begin{equation}
			\label{eq_gropa11}
			\inf\limits_{x\in\hilbertX} f(x)+g(x)=\inf\limits_{x\in\hilbertX} f(x)+g_{\Phi}^{**}(x)=\inf\limits_{x\in\hilbertX} \sup\limits_{\varphi\in\Phi}{\mathcal L}(x,\varphi)=\kappa.
		\end{equation}
		%{\color{cyan} (\forall\ \varepsilon>0)\  \text{there exists an %$\varepsilon$-solution}\   x_{\varepsilon}\in\hilbertX \  \text{ to problem  %\eqref{lagprimal} with  } g(x_{\varepsilon})=g^{**}(x_{\varepsilon})},
		%$$
		then 
		$(2)$ implies $(1)$.
		\item If  $\Phi+\Phi\subset \Phi$, then $(1)$ implies $(2)$.
	\end{enumerate}
\end{theorem}
\begin{proof} 
	$(2)\Rightarrow (1).$
	Let 
	%$\kappa:=\inf\limits_{x\in \hilbertX} \sup\limits_{\varphi\in \Phi} {\mathcal L}(x,\varphi)<+\infty$, 
	$\varepsilon>0$. 
	%Assume condition $(2)$ (intersection property) holds. 
	By Lemma \ref{key_lemma},  there exists $t_{0}\in[0,1]$
	%\begin{equation} 
	%\label{key_ineq} 
	$t_{0}\varphi_{1}+(1-t_{0})\varphi_{2}\ge\alpha:=\kappa-\varepsilon\ \ \forall\ \ x\in \hilbertX.$
	%\end{equation}
	Hence, 
	%\begin{equation} 
	$$
	{\cal L}(x,t_{0}\psi_{1}+(1-t_{0})\psi_{2})\ge t_{0}{\cal L}(x,\psi_{1})+(1-t_{0}){\cal L}(x,\psi_{2})\ge t_{0}\varphi_{1}+(1-t_{0})\varphi_{2}
	$$
	%\end{equation}
	for all $x\in \hilbertX$ and
	\begin{equation} 
		\label{key_ineq2n} 
		{\cal L}(x,\psi_{0})\ge \kappa-\varepsilon=(\varepsilon+\kappa)-2\varepsilon \ \ \forall\ \ x\in\hilbertX,
	\end{equation}
	where $\psi_{0}:=t_{0}\psi_{1}+(1-t_{0})\psi_{2}\in\Phi$ (in view of  the convexity of $\Phi$).
	%There exists $\bar{x}\in\hilbertX$ such that
	%\begin{equation} \end{equation}
	%By Theorem \ref{conju} $(v)$, since $g$ is $\Phi$-convex on $\hilbertX$, 
	By assumption \eqref{eq_gropa11}, and Theorem \ref{conju}, $(iii)$, there exists $\bar{x}\in\hilbertX$ satisfying
	\begin{equation} 
		\label{key_ineq21a} 
		\varepsilon+\kappa>f(\bar{x})+g(\bar{x})\ge \sup\limits_{\varphi\in\Phi}{\mathcal L}(\bar{x},\varphi)=f(\bar{x})+g^{**}_{\Phi}(\bar{x}).
	\end{equation}
	Moreover, $\varepsilon +f(\bar{x})+g^{**}_{\Phi}(\bar{x})\ge\varepsilon+\kappa>f(\bar{x})+g(\bar{x})$, hence  $g^{**}_{\Phi}(\bar{x})\ge g(\bar{x})-\varepsilon$ and
	\begin{equation} 
		\label{key_ineq21aa} 
		\varepsilon+\kappa>f(\bar{x})+g(\bar{x})\ge \sup\limits_{\varphi\in\Phi}{\mathcal L}(\bar{x},\varphi)=f(\bar{x})+g^{**}_{\Phi}(\bar{x})\ge f(\bar{x})+g(\bar{x})-\varepsilon.
	\end{equation}
	%$\bar{x}$ can be chosen so as to have $g(\bar{x})=g^{**}(\bar{x})$, hence

	%\begin{equation} 
	%\label{key_ineq22} 
	%\sup\limits_{\varphi\in\Phi}{\mathcal %L}(\bar{x},\varphi)=f(\bar{x})+g^{**}(\bar{x})=f(\bar{x})+g(\bar{x}).
	%\end{equation}
	By 
	%\eqref{key_ineq2n},\eqref{key_ineq21}, \eqref{key_ineq22}, 
	\eqref{key_ineq2n}, \eqref{key_ineq21a}, \eqref{key_ineq21aa}  for all $x\in \hilbertX$,
	%\begin{equation} 
	%\label{key_ine 223}
	$$
	\begin{array}{l}
		(a). \ \ {\cal L}(x,\psi_{0})\ge\sup\limits_{\varphi\in\Phi}{\cal L}(\bar{x},\varphi)-2\varepsilon\ge  {\cal L}(\bar{x},\psi_{0})-3\varepsilon\\
		(b).  \ \ {\cal L}(x,\psi_{0})\ge f(\bar{x})+g(\bar{x})- 3\varepsilon.\\
	\end{array}
	$$
	In particular, by $(b)$,
	$$
	(b').\ \ \ \ {\cal L}(\bar{x},\psi_{0})\ge f(\bar{x})+g(\bar{x})-3\varepsilon.
	$$
	Since $\Phi=-\Phi$, hence $-\psi_{0}\in\Phi$ and  by $(a)$,
	\begin{equation} 
		\label{key_ineq55} 
		f(x)+\psi_{0}(x)-g^{*}(\psi_{0})={\cal L}(x,\psi_{0})\ge  f(\bar{x})+\psi_{0}(\bar{x})-g^{*}(\psi_{0})-3\varepsilon\ \ \forall\ \ x\in\hilbertX.
	\end{equation}
	which shows that $-\psi_{0}\in\partial^{3\varepsilon} f(\bar{x})$.
	By $(b')$,
	$$ 
	{\cal L}(\bar{x},\psi_{0})\ge f(\bar{x})+g(\bar{x})-3\varepsilon
	\ \ \ i.e.\ \ \ 
	f(\bar{x})+\psi_{0}(\bar{x})-g^{*}(\psi_{0})\ge f(\bar{x})+g(\bar{x})-3\varepsilon
	$$
	which gives  $\psi_{0}(\bar{x})-g^{*}(\psi_{0})\ge g(\bar{x})-3\varepsilon$. 
	
	By Theorem \ref{conju} $(ii)$,
	the latter is equivalent to
	$\psi_{0}\in\partial^{3\varepsilon}g(\bar{x}).$
	
	This means that 
	$$
	0\in\bigcap_{\varepsilon>0}(\partial^{\varepsilon}f+\partial^{\varepsilon}g)(\hilbertX)
	$$
	i.e. $(1)$ holds.

	$(1)\Rightarrow (2)$. Let $\alpha < \inf\limits_{x\in \hilbertX} \sup\limits_{\varphi\in \Phi} {\mathcal L}(x,\varphi)$ and take any $\beta$ such that $\alpha <\beta< \inf\limits_{x\in \hilbertX} \sup\limits_{\varphi\in \Phi} {\mathcal L}(x,\varphi)$. Let $\varepsilon=\beta-\alpha>0$. 
	By   assumption,  there exist $\bar{x}\in \hilbertX$ and $\bar{\varphi}\in \partial^{\varepsilon}_{\Phi}f(\bar{x})$ and  $\tilde{\varphi}\in \partial^{\varepsilon}_{\Phi}g(\bar{x})$ such that
	\begin{equation}
		\label{hoaeps}
		\bar{\varphi}+\tilde{\varphi}=0.
	\end{equation}
	Since $\bar{\varphi}\in \partial^{\varepsilon}_{\Phi}f(\bar{x})$ 
	the following inequality holds for all $x\in\hilbertX$
	$$
	f(x)-f(\bar{x})\geq \bar{\varphi}(x)- \bar{\varphi}(\bar{x})-\varepsilon, 
	$$
	which is equivalent to
	$$
	f(x)+(-\bar{\varphi}(x)) -g_{\Phi}^*(-\bar{\varphi})-(f(\bar{x})+(-\bar{\varphi}(\bar{x}))-g_{\Phi}^*(-\bar{\varphi}))\geq 0-\varepsilon \ \ \ \ \forall \ \ x\in \hilbertX.
	$$
	Equivalently,
	\begin{equation}
		\label{hoa1eps1}
		{\mathcal L}(x,-\bar{\varphi})-{\mathcal L}(\bar{x}, -\bar{\varphi})\geq-\varepsilon.
	\end{equation}
	%i.e. $0\in \partial_{\Phi}L(\cdot, -\bar{\varphi})(\bar{x})$.
	By the equality \eqref{hoaeps} we have $\tilde{\varphi}=-\bar{\varphi}$ and the inequality \eqref{hoa1eps1} takes the form
	\begin{equation}
		\label{hoa2eps1}
		{\mathcal L}(x,\tilde{\varphi})\geq {\mathcal L}(\bar{x}, \tilde{\varphi})-\varepsilon \ \ \ \  \forall \ \ \ x\in  \hilbertX.
	\end{equation}
	Since $\tilde{\varphi}\in \partial^{\varepsilon}_{\Phi}g(\bar{x})$, the  inequality  $g^*_{\Phi}(\tilde{\varphi})+g(\bar{x})\leq \tilde{\varphi}(\bar{x})+\varepsilon$ holds (see Theorem \ref{conju}(ii)).
	We have
	$$
	{\mathcal L}(\bar{x}, \tilde{\varphi})= f(\bar{x})+\tilde{\varphi}(\bar{x})-g^*_{\Phi}(\tilde{\varphi}) \geq f(\bar{x})+g(\bar{x})-\varepsilon\geq \inf_{x\in  \hilbertX}\{ f(x)+g(x)\}-\varepsilon
	$$
	The above inequality together with \eqref{hoa2eps1} and  \eqref{eq_gropa11}    give 
	\begin{equation}
		\label{lag1}
		{\mathcal L}(x,\tilde{\varphi})\geq\inf_{x\in  \hilbertX}\{ f(x)+g(x)\}-\varepsilon= \inf\limits_{x\in \hilbertX} \sup\limits_{\varphi\in \Phi} {\mathcal L}(x,\varphi) -\varepsilon\ \ \ \  \forall \ \ \ x\in \hilbertX.
	\end{equation}
	% Since $\varepsilon>0$ was chosen arbitrarily, above inequality holds for all $\varepsilon>0$, by letting $\varepsilon \downarrow 0$, we have
	% $$
	% L(x,\tilde{\varphi})\geq \inf\limits_{x\in X} \sup\limits_{\psi\in \Psi} {\mathcal L}(x,\psi) \ \ \ \  \forall \ \ \ x\in X.
	% $$
	
	By the inequality \eqref{lag1} we have
	$$
	{\cal L}(x,\tilde{\varphi})\geq \beta-\varepsilon=\beta-\beta+\alpha=\alpha.
	$$
	Let $\varphi_1\equiv \alpha$ then $ \varphi_1\in \text{supp}\, {\cal L}(\cdot,\tilde{\varphi})$ and  $[\varphi_1<\alpha]=\emptyset$. Let $\varphi_2 \in \text{supp}\, {\cal L}(\cdot,\bar{\varphi})$, then $\varphi_1$, $\varphi_2$ have the intersection property at the level $\alpha$.
\end{proof}

%\noindent
\begin{remark}
	\begin{enumerate} 
		\item  Let $\hilbertX$ be a topological vector space equipped with closed convex pointed cone $S$ which induces the ordering relation: $x\le y\ \Leftrightarrow \ y-x\in S$ The family of functions $L:=\{\ell_{y}:\hilbertX\rightarrow\mathbb{R}\ \mid\ y\in\hilbertX\}$ defined as
		$$
		\ell_{y}(x)=\max\{\lambda\ge 0\ \lambda y\le x\}
		$$
		It was shown in \cite{Mohebi13}, (see also \cite{Eberhard10}) that $0\in L$,  $L+L\subset L$, and $L=-L$ and a function $f:\hilbertX\rightarrow[0,,+\infty]$ is increasing positive homogeneous (IPH) if and only if $f$ is $L$-convex.  $L$-conjugate dual \eqref{condual3} coincides with $L$-infimal convolution dual \eqref{hoabui}.
		\item   By Theorem \ref{th_implication}, if $\Phi$ is convex, $0\in\Phi$, $\Phi+\Phi\subset\Phi$ and $\Phi=-\Phi$,
		$$
		\text {condition  }\eqref{bui}\ \Leftrightarrow \ \text {condition  }\eqref{eq-111}+ \text {equality}\eqref{eq_gropa11}.
		$$
		\item If $0\in\Phi$, $\Phi+\Phi\subset\Phi$ and $\Phi=-\Phi$, then \eqref{condual3} is equivalent to \eqref{hoabui} (see Corollary \ref{corzeroduality} below) and  Theorem \ref{bui3.6} is stronger than Theorem \ref{lagmin1} since the convexity of $\Phi$ is not required in Theorem \ref{bui3.6}.
		\item If $0\in\Phi$, $\Phi+\Phi\subset\Phi$ and $\Phi\neq-\Phi$, (e.g.,  $\Phi=\Phi_{lsc}$ ) the $\Phi$-infimal convolution dual \eqref{hoabui} is defined but is not equivalent to \eqref{condual3} and it may happen that zero duality gap holds for \eqref{condual3} but not for \eqref{hoabui} (see Example \ref{example1} ).
		
	\end{enumerate}
\end{remark}

\begin{example}
	\label{example1}
	Let $\Phi=\Phi_{lsc}$. Let $g(x)=-x^2$ and $f(x)=2x^2$. It is easy to see that $\inf\limits_{x\in \hilbertX}(f(x)+g(x))=0$. For every $\varepsilon>0$, the elements of the set $\partial^{\varepsilon}_{\Phi}g$ are of the form $(a,b)\in \mathbb{R}\times \mathbb{R}$ with $a\geq 1$ and some $b\in \mathbb{R}$, this means that  
	$$
	0\notin \bigcap_{\varepsilon>0}(\partial^{\varepsilon}_{\Phi}f+ \partial^{\varepsilon}_{\Phi}g)(\hilbertX).
	$$
	On the other hand, 
	$$
	{\cal L}(x,1,0)=2x^2-x^2-\sup_{x\in X}\{-x^2+x^2\}=x^2.
	$$
	This means that $\hat{\varphi}\equiv0$ belongs to the set $\text{supp}{\cal L}(x,1,0)$. We have
	$$
	{\cal L}(x,3,0)=2x^2-3x^2-\sup_{x\in X}\{-3x^2+x^2\}=-x^2,
	$$
	and the set $\text{supp}{\cal L}(x,3,0)\neq \emptyset$. Hence the functions $\hat{\varphi}$ and any other $\varphi\in \text{supp}{\cal L}(x,3,0)$ have the intersection property at every level $\alpha<0$.
	
	For any $a\ge 0$, $b\in\mathbb{R}$ and $\tilde{f}=(2-a)x^{2}$ we have
	$$
	\begin{array}{l} 
		g^{*}(a,b)=\sup_{x\in\mathbb{R}}\{bx-ax^{2}+x^{2}\}=\left\{\begin{array}{lll}
			%+\infty&\mbox{for}&0\le a\le 1,\ b\neq 0\\
			%+\infty&\mbox{for}&a= 0,\ \\
			+\infty&\mbox{for}&0\le a< 1 \ \\
			+\infty&\mbox{for}&a=1\ b\neq 0 \ \\
			0&\mbox{for}&a=1\ b= 0,\ \\
			\frac{b^{2}}{4(a-1)}&\mbox{for}&a>1\\
		\end{array}
		\right.\\
		\tilde{f}^{*}(0,-b)=\sup_{x\in\mathbb{R}}\{-bx+ax^{2}-2x^{2}\}=\left\{\begin{array}{lll}
			%\frac{b^{2}}{8}&\mbox{for}&a=0,\\
			\frac{-b^{2}}{4(a-2)}&\mbox{for}&0\le a< 2\\
			0&\mbox{for}&a=2,\ b=0\\
			+\infty&\mbox{for}&a=2,\ b\neq 0\\
			+\infty&\mbox{for}&a>2,\\
		\end{array}
		\right.\\
		f^{*}(0,b)=\sup_{x\in\mathbb{R}}\{bx-2x^{2}\}=\frac{b^{2}}{8}.
	\end{array}
	$$
	In consequence,
	$$
	val(ICD)=\sup_{b\in\mathbb{R}}-f^{*}(0,-b) -g^{*}(0,b)=\sup_{b\in\mathbb{R}}-\frac{b^{2}}{8}-\infty=-\infty
	$$
	and
	$$
	val(CD)\begin{array}[t]{l}
		=\sup_{(a,b)\in\mathbb{R}_{+}\times\mathbb{R}}-\tilde{f}^{*}(0,-b)-g^{*}(a,b)\\
		=\max\{0,\sup_{1<a<2\ b\in\mathbb{R}}\frac{b^{2}}{4(a-2)}-\frac{b^{2}}{4(a-1)}\}\\ \\
		=\max\{0,\sup_{1<a<2\ b\in\mathbb{R}}\frac{b^{2}}{4}\frac{1}{(a-1)(a-2)}\}\\
		=0
	\end{array}
	$$
\end{example}

\begin{corollary}
	\label{corzeroduality}
	Let $\hilbertX$  be  a  real vector space.  
	Let $g: \hilbertX \rightarrow (-\infty,+\infty]$ be $\Phi$-convex. 
	If  $\Phi+\Phi\subset\Phi$, $0\in \Phi$, $\Phi=-\Phi$ and $\Phi$ is a convex set, 
	then the intersection property \eqref{eq-11} is equivalent to \eqref{bui}.
	Consequently,  \eqref{hoabui} is equivalent to \eqref{condual3a}, i.e. 
	\begin{equation} 
		\label{hoabuisym}
		val(ICD)
		=\sup\limits_{\varphi\in\Phi}-f_{\Phi}^{*}(-\varphi)-g_{\Phi}^{*}(\varphi)=val(CD^{sym}).
	\end{equation}
\end{corollary}
\begin{proof}
	Follows directly from Theorem \ref{th_implication}.
\end{proof}

\section{ $\Phi$-Karush-Kuhn-Tucker conditions.}
In this section we provide a characterisation of solutions to $(P)$ and $(CD)/(CD^{sym})$ in terms of the $\Phi$-Karush-Kuhn-Tucker conditions.
%the conditions for the strong duality for the problems  \eqref{problem} and \eqref{gencondual} i.e the conditions ensuring that the problem  \eqref{gencondual} has a solution. 

Let $\hilbertX$ be a real  vector space. Consider problem \eqref{problema}
\begin{equation}
	\label{problema}
	\tag{P}
	\text{Min}_{x\in \hilbertX}\ \ f(x)+g(x). 
\end{equation}
where $f,g:\hilbertX\rightarrow(-\infty,+\infty]$
are $\Phi$-convex. 

The existence of solutions to the dual problem  $(CD^{sym})$ was  investigated in \cite{JRW} and the following result was proved.
\begin{proposition}(\cite{JRW}, Corollary 5.2)
	Let $\Phi$ be an additive and symmetric set of elementary functions, i.e. $-\varphi\in \Phi$ if $\varphi\in \Phi$. Assume that $f$  and $g$ are $\Phi$-convex.
	If the mapping $\text{supp}(\cdot, \Phi)$  is additive in $f, g$, then
	there exists $\varphi^{*}\in \Phi$,  such that
	$$
	\inf_{x\in X}\{f(x)+g(x)\}=-f_{\Phi}^{*}(\varphi^{*})-g_{\Phi}^{*}(-\varphi^{*})=\sup_{\varphi\in\Phi}\{-f_{\Phi}^{*}(\varphi)-g_{\Phi}^{*}(-\varphi) \}
	$$
\end{proposition}

\begin{definition}
	Let $\Phi$ be symmetric i.e.  $\Phi=-\Phi$. We say that $x^{*}\in \hilbertX$ and $\varphi^{*}\in\Phi$ satisfy the {\em $\Phi$-Karush-Kuhn-Tucker conditions} $(KKT)$ for  the pair of dual problems $(P)$ and  $(CD^{sym})$   if
	\begin{equation} 
		\label{solvability2} 
		\tag{KKT}
		-\varphi^*\in\partial_{\Phi} f(x^{*}),\ \ \ x^{*}\in\partial_{\hilbertX} g_{\Phi}^{*}(\varphi^{*}).
	\end{equation}
\end{definition}

\begin{theorem}
	\label{kkt}
	Let $\hilbertX$ be a vector space and $\Phi$ be a symmetric set of elementary functions. 
	%{\color{blue} $-\varphi\in\Phi$ if $\varphi\in\Phi$}. 
	Let $f,g:\hilbertX\rightarrow(-\infty,+\infty]$ be $\Phi$-convex functions. Let $x^{*}\in\hilbertX$ and $\varphi^{*}\in\Phi$.
	
	The following conditions are equivalent.
	\begin{description}
		\item[(i)] $x^{*}$ and $\varphi^{*}$ are solutions to $(P)$ and $(CD^{sym})$, respectively, i.e.
		\begin{equation} 
			\label{solvability22} 
			\inf\limits_{x\in\hilbertX} (f(x)+g(x))=f(x^{*})+g(x^{*})=\sup_{\varphi\in\Phi}(-f_{\Phi}^{*}(-\varphi)-g_{\Phi}^{*}(\varphi)) =-f_{\Phi}^{*}(-\varphi^*)-g_{\Phi}^{*}(\varphi^*) 
		\end{equation}
		\item[(ii)]  $x^{*}$ and $\varphi^{*}$  satisfy the $\Phi$-KKT conditions, i.e.,
		\begin{equation} 
			\label{solvability223} 
			-\varphi^*\in\partial_{\Phi} f(x^{*}),\ \ \ x^{*}\in\partial_{\hilbertX} g_{\Phi}^{*}(\varphi^{*}).
		\end{equation}
	\end{description}
\end{theorem}
\begin{proof}
	Assume that  \eqref{solvability22} holds, i.e.
	\begin{equation}
		\label{equ123}
		f(x^*)+g(x^*)=-f_{\Phi}^*(-\varphi^*)-g_{\Phi}^*(\varphi^*).
	\end{equation}
	By Theorem \ref{conju} and the $\Phi$-convexity of $g$, \eqref{solvability22} is equivalent to
	$$
	f(x^*)+g^{**}_{\Phi}(x^*)=-f_{\Phi}^*(-\varphi^*)-g_{\Phi}^*(\varphi^*).
	$$
	%Which is equivalent to 
	By the definition of $g^{**}_{\Phi}(x^*)$,
	$f(x^*)+\varphi^*(x^*)-g_{\Phi}^*(\varphi^*)\leq -f_{\Phi}^*(-\varphi^*)-g_{\Phi}^*(\varphi^*)$, i.e.,
	$$
	f(x^*)+f_{\Phi}^*(-\varphi^*)\leq -\varphi^*(x^*).
	$$
	This, together with the Fenchel-Moreau inequality yields to  $
	f(x^*)+f_{\Phi}^*(-\varphi^*)= -\varphi^*(x^*)
	$ i.e. $-\varphi^*\in\partial_{\Phi} f(x^{*})$.
	
	Analogously, by replacing in  \eqref{equ123} function $f$ with $f_{\Phi}^{**}(x^*)$ we obtain $x^*\in \partial_{\hilbertX}g^*_{\Phi}(\varphi^*)$.

	Assume now that the conditions \eqref{solvability223} hold. By the Proposition \ref{you}, 
	\begin{equation}
		\label{eqf1}
		-\varphi^*(x^*)=f_{\Phi}^*(-\varphi^*)+f(x^*)
	\end{equation}
	and
	\begin{equation}
		\label{eqg1}
		\varphi^*(x^*)=g_{\Phi}^*(\varphi^*) +g(x^*).
	\end{equation}
	%By adding the above qualities \eqref{eqf1} and \eqref{eqg1} we get
	Hence,
	$$
	f(x^*)+g(x^*)=-f_{\Phi}^*(-\varphi^*)-g_{\Phi}^*(\varphi^*).
	$$
	From  \eqref{eqg1} we get  $\varphi^*\in \partial_{\Phi} g(x^*)$, which means that $0\in \partial_{\Phi} f(x^*)+\partial _{\Phi}g(x^*)\subset \partial_{\Phi} (f+g)(x^*)$. 
	From the equality \eqref{eqf1} we have $x^*\in \partial_{\hilbertX} f_{\Phi}^*(-\varphi^*)$, this, together  with the assumption that $x^{*}\in\partial_{\hilbertX} g_{\Phi}^{*}(\varphi^{*})$, yields to
	$$
	f_{\Phi}^*(-\varphi)+g_{\Phi}^*(\varphi)-f_{\Phi}^*(-\varphi^*)-g_{\Phi}^*(\varphi^*)\geq -\varphi(x^*)+\varphi^*(x^*)+\varphi(x^*)-\varphi^*(x^*)\ \ \ \forall \ \ \ \varphi\in \Phi.
	$$
	Equivalently,
	$-f_{\Phi}^*(-\varphi)-g_{\Phi}^*(\varphi)\leq -f_{\Phi}^*(-\varphi^*)-g_{\Phi}^*(\varphi^*) $ for all $\varphi\in \Phi$
	which means that $-f_{\Phi}^*(-\varphi^*)-g_{\Phi}^*(\varphi^*)=\sup\limits_{\varphi\in\Phi}(-f_{\Phi}^{*}(-\varphi)-g_{\Phi}^{*}(\varphi))$. This completes the proof. 
	
\end{proof}

\subsection{KKT for $\Phi= \Phi_{lsc}$}

Let $\hilbertX$ be a Hilbert space. In the present subsection we prove a variant of Theorem \ref{kkt} with $f$ and $g$ which are $\Phi_{lsc}$-convex, where \begin{equation} 
	\label{philsc1}
	\Phi_{lsc}=\{\varphi:\hilbertX\rightarrow\mathbb{R}\ \mid\ \varphi(x):=-a\|x\|^{2}+\langle v,x\rangle +c,\ a\ge 0,\ v\in\hilbertX,\ c\in\mathbb{R}\}.
\end{equation}
The set $\Phi_{lsc}$ is nonsymmetric and forms a non  pointed cone with the lineality space $L=\hilbertX$. By Proposition 6.3 of \cite{rubbook}, the class of $\Phi_{lsc}$-convex functions defined on Hilbert space $\hilbertX$ coincides with the class of all lower semicontinuous functions minorized by a function $\varphi\in\Phi_{lsc}$. Clearly, the class $\Phi_{lsc}$ is additive and the sum $f+g$ of any  $\Phi_{lsc}$-convex functions $f$ and $g$ is a $\Phi_{lsc}$ function.

Recall that the $\Phi_{lsc}$-conjugate dual to problem $(P)$ with $\Phi_{lsc}$-convex functions $f$ and $g$ has the form \eqref{conjugate_dual1}
\begin{equation} 
	\label{conjugate_dual11}
	\tag{$CD^{lsc}$}
	%\begin{equation*}
	\text{Max\,}_{(a,w)\in\Phi_{lsc}} -(\tilde{f}_{(a,w))})_{\Phi_{lsc}}^{*}(0, -w)-g_{\Phi_{lsc}}^{*}(a,w),
	%\end{equation*}
\end{equation}
where  functions of the form $\psi_{1}(x):=-a\|x\|^{2}+\langle w,x\rangle$  are identified with  pairs $(a,w)$, $a\ge 0$, $w\in\hilbertX$ and $\tilde{f}_{\psi_{1}}(x)=\tilde{f}_{(a,w)}(x):=f(x)-a\|x\|^{2}$.

\begin{theorem}
	\label{kktlsc}
	Let $\hilbertX$ be a Hilbert space. Let $f,g:\hilbertX\rightarrow(-\infty,+\infty]$ be  $\Phi_{lsc}$-convex functions. 
	
	In order that $x^{*}\in\hilbertX$ and $\varphi^{*}=(a^{*},w^{*})\in\Phi_{lsc}$  be  such that
	\begin{equation} 
		\label{solvabilitylsc1} 
		\begin{array}{l}
			\inf\limits_{x\in\hilbertX} (f(x)+g(x))=f(x^{*})+g(x^{*})=\sup\limits_{(a,w)\in\Phi_{lsc}}(-(\tilde{f}_{(a,w)})_{\Phi_{lsc}}^{*}(0,-w)-g_{\Phi_{lsc}}^{*}(0,w))\\ =-(\tilde{f}_{(a^{*},w^{*})})_{\Phi_{lsc}}^{*}(0,-w^{*})-g_{\Phi_{lsc}}^{*}(a^{*},w^{*}), 
		\end{array}
	\end{equation}
	i.e. $x^{*}\in\hilbertX$ solves $(P)$ and $\varphi^{*}=(a^{*},w^{*})\in\Phi_{lsc}$ solves \eqref{conjugate_dual11}
	it is necessary and sufficient that $x^{*}\in \hilbertX$ and $\varphi^{*}\in\Phi_{lsc}$ satisfy the $\Phi$-Karush-Kuhn-Tucker conditions \eqref{solvability2},
	\begin{equation} 
		\label{solvabilitylsc} 
		(0,-w^{*})\in\partial_{lsc}\tilde{f}(x^{*}),\ \ \ x^{*}\in\partial_{\hilbertX}g_{\Phi_{lsc}}^{*}(\varphi^{*})
	\end{equation}
	where $\partial_{lsc}$ denotes the $\Phi_{lsc}$-subgradient.
\end{theorem}
\begin{proof}
	By Proposition \ref{you}, the inclusions in formulae \eqref{solvabilitylsc} are respectively  equivalent to the following  equalities
	\begin{equation}
		\label{eqflsc}
		-\langle w^{*},x^{*}\rangle=(\tilde{f}_{(a^{*},w^{*})})_{\Phi_{lsc}}^*(0,-w^{*})+\tilde{f}(x^*)=(\tilde{f}_{(a^{*},w^{*})})_{\Phi_{lsc}}^*(0,-w^{*})+f(x^*)-a^{*}\|x^{*}\|^{2},
	\end{equation}
	where $\tilde{f}_{(a^{*},w^{*})}(x^{*}):=f(x^{*})-a^{*}\|x^{*}\|^{2}$ and
	$\varphi^*(x^*)=-a^{*}\|x^{*}\|^{2}+\langle w^{*},x^{*}\rangle=g_{lsc}^*(\varphi^*) +g(x^*).$
	Consequently,
	%By adding the  equalities \eqref{eqf} and \eqref{eqg} we get
	$$
	(\tilde{f}_{(a^{*},w^{*})})_{\Phi_{lsc}}^{*}(0,-w^{*})+f(x^*)=-g_{\Phi_{lsc}}^*(\varphi^*)-g(x^*)
	$$
	and, in view of \eqref{pdineq}, we obtain
	\eqref{solvabilitylsc1}.
	
	Assume now that  \eqref{solvabilitylsc1} holds with $\varphi^{*}(x):=-a^{*}\|x\|^{2}+\langle w^{*},x\rangle$, i.e.
	\begin{equation}
		\label{equ1234}
		f(x^*)+g(x^*)=-(\tilde{f}_{(a^{*},w^{*})})_{\Phi_{lsc}}^*(0,-w^{*})-g_{\Phi_{lsc}}^*(\varphi^*),
	\end{equation}
	where 
	$$
	(\tilde{f}_{(a^{*},w^{*})})_{\Phi_{lsc}}^*(0,-w^{*})=\sup\limits_{x\in \hilbertX} -\langle w^{*},x\rangle-(f(x)+a^{*}\|x\|^{2})
	$$
	and, as previously,  $\tilde{f}_{\varphi}(\cdot)=\tilde{f}_{(a,w)}(\cdot)=f(\cdot)-a\|\cdot\|^{2}$ for any $\varphi\in\Phi$.
	By Theorem \ref{conju} and the $\Phi$-convexity of $g$, \eqref{equ1234} is equivalent to
	$$
	f(x^*)+g_{\Phi_{lsc}}^{**}(x^*)=-(\tilde{f}_{(a^{*},w^{*})})_{\Phi_{lsc}}^*(0,-w^{*})-g_{\Phi_{lsc}}^*(\varphi^*).
	$$
	%Which is equivalent to 
	By  definition of $g_{\Phi_{lsc}}^{**}(x^*)$,
	$f(x^*)+\varphi^*(x^*)-g_{\Phi_{lsc}}^*(\varphi^*)\leq -(\tilde{f}_{(a^{*},w^{*})})_{\Phi_{lsc}}^*(0,-w^{*})-g_{\Phi_{lsc}}^*(\varphi^*)$, i.e.,
	$$
	f(x^*)+(\tilde{f}_{(a^{*},w^{*})})_{\Phi_{lsc}}^*(0,-w^{*})\leq -\varphi^*(x^*).
	$$
	Hence, $\tilde{f}(x^*)+(\tilde{f}_{(a^{*},w^{*})})_{\Phi_{lsc}}^*(0,-w^{*})\leq -\langle w^*,x^*\rangle$ and,
	by Fenchel-Moreau inequality, 
	$\tilde{f}_{(a^{*},w^{*})}(x^*)+(\tilde{f}_{(a^{*},w^{*})})_{\Phi_{lsc}}^*(0,-w^{*})= -\langle w^*,x^*\rangle$
	i.e. $(0,-w^{*})\in\partial_{lsc} \tilde{f}_{(a^{*},w^{*})}(x^{*})$.
	
	Analogously, by \eqref{equ1234},
	$\tilde{f}_{(a^{*},w^{*})}(x^*)+a^{*}\|x\|^{2}+g(x^*)=-(\tilde{f}_{(a^{*},w^{*})})_{\Phi_{lsc}}^*(0,-w^{*})-g_{\Phi_{lsc}}^*(\varphi^*)$, and
	since $\tilde{f}$ is $\Phi_{lsc}$ convex, $\tilde{f}_{(a,w)}=(\tilde{f}_{(a,w)})_{\Phi_{lsc}}^{**}$, and
	$$
	-(\tilde{f}_{(a^{*},w^{*})})^{*}_{\Phi_{lsc}}(0,-w^{*})-\langle w^{*},x^{*}\rangle +a^{*}\|x^{*}\|^{2}+g(x^{*})\le -(\tilde{f}_{(a^{*},w^{*})})_{\Phi_{lsc}}^{*}(0,-w^{*})-g_{\Phi_{lsc}}^{*}(\varphi^{*}),
	$$
	i.e. $g(x^{*})+g_{\Phi_{lsc}}^{*}(\varphi^{*})\le \varphi^{*}(x^{*})$ which, together with Fenchel-Moreau inequality, gives
	$x^*\in \partial_{\hilbertX}g_{\Phi_{lsc}}^*(\varphi^*)$.
\end{proof}

The example below illustrates Theorem \ref{kktlsc}.

\begin{example}
	Let $\hilbertX=\mathbb{R}$. Let $f,g:\mathbb{R}\rightarrow\mathbb{R}$ be given by the following formulas
	$$
	f(x)=\left \{ 
	\begin{matrix}
		2(x-1)^2 & \ \ \text{for} \ \ x\geq 0\cr
		2(x+1)^2 & \ \ \ \text{for} \ \ x<0\cr
	\end{matrix}
	\right., \ \ \ \ \ \ \ g(x)=-x^2.
	$$
	It is easy to see that $f$ and $g$ are nonconvex but $\Phi_{lsc}$-convex, where the set $\Phi_{lsc}$ is defined by \eqref{philsc}.
	% Let us note that the class $\Phi_{lsc}$ is not symmetric. Define the class $\bar{\Phi}_{lsc}$ as follows
	% $$
	% \bar{\Phi}_{lsc}:=\{\varphi:\mathbb{R}\rightarrow\mathbb{R}\ \mid\ \varphi(x)=-ax^{2}+vx +c,\ a, v,c\in\mathbb{R}\}.
	% $$
	Functions $\varphi\in \Phi_{lsc}$ such that $c=0$ will be 
	identified with pairs $(a,v)$, $a,v\in\mathbb{R}$.
	
	Consider the problem
	$$
	\inf_{x\in \mathbb{R}}\{f(x)+g(x)\}.
	$$
	Let $\varphi^*(x)=-x^2$, which we identify with the pair $(1,0)$, it is easy to see that
	$$
	f(x)-f(2)\geq x^2-(2)^2
	$$
	i.e. $(-1,0)\in \partial_{lsc}f(2)$. Now we show that $2\in\partial_{\hilbertX}g_{lsc}^*(1,0)$ i.e that the KKT conditions hold and this means,  that $\varphi^*(x)=-x^2$ and $x^*=2$ are the solutions of the dual and primal problems, respectively. By simple calculations we get
	$$
	g^*(a,b)=\sup_{x\in\mathbb{R}}\{ -ax^2+bx+x^2\}=\left \{ 
	\begin{matrix}
		+\infty & \ \ \text{for} \ \ a\leq 1,b\neq 0\cr
		0 & \ \ \ \text{for} \ \ a=1,b=0\cr
		\frac{-b^2}{4(1-a)}&	\ \ \ \text{for} \ \ a>1,b\neq0\end{matrix}
	\right.,
	$$
	and it is easy to see that $2\in\partial_{\hilbertX}g_{lsc}^*(1,0)$ and $(0,0)\in \partial_{lsc}\tilde{f}(2)$, where $\tilde{f}_{(a,v)}=f-a\|\cdot\|^{2}$. 
\end{example}

\section{Conclusions}
In conclusion, Theorem \ref{lagmin1} provides sufficient and necessary conditions for zero duality gap for primal \eqref{problem},  $\Phi$-Lagrangian \eqref{lagprimal}, $\Phi$-Lagrangian dual \eqref{lagdual} and $\Phi$-conjugate dual \eqref{condual3} problems,  for a suitably defined $\Phi$-Lagrangian function  when $0\in \Phi$, and $\Phi$ is a convex set. Theorem \ref{lagmin1}, together with Theorem \ref{kkt} and Theorem \ref{kktlsc} reveal the importance of properties of the elementary functions $\Phi$ in general duality theory.

\bibliographystyle{siam} 
\bibliography{bibn}

\end{document}